\documentclass[12pt]{amsart}
\usepackage{amsmath,amscd,amsthm,amsfonts, amssymb,amsxtra, accents}
\usepackage[all,cmtip]{xy}
\usepackage{tikz}
\usetikzlibrary{positioning,calc,decorations.markings}
\usepackage{xparse,etoolbox}
\usepackage[pagewise]{lineno}
\usepackage{color}
\usepackage{lipsum}

\usepackage{hyperref}
  \hypersetup{colorlinks=true,citecolor=blue}
  \usepackage{graphicx, epstopdf}
 \usepackage[font=scriptsize]{caption} 
\CDat
\newtheorem{thm}{Theorem}[section]  
\newtheorem*{un-no-thm}{Theorem}
\newtheorem{cor}[thm]{Corollary}     
\newtheorem{lem}[thm]{Lemma}         
\newtheorem{prop}[thm]{Proposition}

\newtheorem{bigthm}{Theorem}

\newtheorem{bigadd}[bigthm]{Addendum}

\theoremstyle{definition}
\newtheorem{defn}[thm]{Definition}   

\theoremstyle{definition}

\theoremstyle{definition}
\theoremstyle{remark}
\newtheorem{rem}[thm]{Remark}
\newtheorem{rems}[thm]{Remarks}

\theoremstyle{Question}

\newtheorem*{mainques}{Question}

\theoremstyle{remark}

\newtheorem{construct}[thm]{Construction}

\newtheorem{hypo}[thm]{Hypothesis}
\newtheorem{notation}[thm]{Notation}
\newtheorem*{acks}{Acknowledgements}

\newtheorem*{out}{Outline}

\newtheorem*{intro-rem}{Remark}
\newtheorem*{intro-rems}{Remarks}

\theoremstyle{remark}
\newtheorem{ex}[thm]{Example}

\newcommand{\blocktheorem}[1]{%
  \csletcs{old#1}{#1}
  \csletcs{endold#1}{end#1}
  \RenewDocumentEnvironment{#1}{o}
    {\par\addvspace{1.5ex}
     \noindent\begin{minipage}{\textwidth}
     \IfNoValueTF{##1}
       {\csuse{old#1}}
       {\csuse{old#1}[##1]}}
    {\csuse{endold#1}
     \end{minipage}
     \par\addvspace{1.5ex}}
}

\blocktheorem{bigthm}

\newcommand{\Mdef}[2]{\newcommand{#1}{\relax\ifmmode #2 \else $#2$\fi}}

\def\:{\colon\!}
\Mdef{\CV}{\mathcal{V}}
\Mdef{\smsh}{\wedge}
\Mdef{\CF}{\mathcal{F}}
\Mdef{\CE}{\mathcal{E}}
\Mdef{\CS}{\mathcal{S}}
\newcommand{\R}{\mathbb{R}}
\newcommand{\Z}{\mathbb{Z}}

\newcommand{\CP}{\mathbb{C}\mathrm{P}}

\newcommand\restr[2]{{ 
  \left.\kern-\nulldelimiterspace 
  #1 
  \vphantom{\big|} 
  \right|_{#2} 
}}

\newcommand{\ra}{\rightarrow}

\title{Exciton scattering via algebraic topology}
\author[M.~J. Catanzaro]{Michael J.\ Catanzaro}
\address{Dept.~ of Mathematics\\ 
                University of Florida \\
                Gainesville, FL 32611}
\email{catanzaro@ufl.edu}

\author[V.~I. Chernyak]{Vladimir Y.\ Chernyak}
\address{Dept.~  of Chemistry\\ 
                Wayne State University\\ 
                Detroit, MI 48202}
\email{chernyak@chem.wayne.edu}

\author[J.~R. Klein]{John R.\ Klein}
\address{Dept.~  of Mathematics\\
                Wayne State University\\
                Detroit, MI 48202}
\email{klein@math.wayne.edu}

\subjclass[2010]{Primary: 92E10, 57R19, 55N45; Secondary: 57N80, 81V55}
\date{\today}
\begin{document}
\maketitle

\begin{abstract} 
  This  paper introduces an intersection theory problem
  for maps into a smooth manifold equipped with a stratification.
 We investigate the problem in the special case
 when the target is the unitary group $U(n)$
 and the domain is a circle.  The first main result is an 
  index theorem that equates a global intersection index
   with a finite sum of locally defined intersection indices. The local indices
  are integers arising from the geometry of the stratification. 
  
  The result is used to study a well-known
  problem in chemical physics, namely, the problem of enumerating the 
  electronic excitations (excitons) of a molecule equipped with
  scattering data. 
  \end{abstract}

\setcounter{tocdepth}{1}
\tableofcontents
\addcontentsline{file}{sec_unit}{entry}


\section{Introduction}\label{sec:intro}

This paper has two objectives. The first objective is
to apply algebraic topology to a problem in chemical physics. The problem, when
translated into topology, is a special case of a more general problem which
is both natural and elementary to state. 
Our other objective is an investigation of the general problem,
which we hope will be of independent interest to mathematicians.

In its barest form, the chemical physics 
problem involves calculating the set of solutions
to a system of linear equations parametrized by the circle. The equations 
are derived from a finite graph equipped with additional structure, 
where the graph is a mathematical representation of a molecule. 
The set of nontrivial solutions of the system of equations is finite 
and the problem is to enumerate them, taking multiplicity into account. 
The solutions are called {\it excitons.} We refer the reader to
the end of this section for some of the historical background and to the chemical
physics paper \cite{Catanzaro:CountingExc} for a more extensive treatment.

To translate the problem into
 algebraic topology, the equations first need to 
be converted to a compact
form.  The requisite form is a parametrized $(+1)$-eigenvalue problem\begin{equation} \label{eqn:eigenvalue-eqn}
\Gamma(z)\psi = \psi\, ,  \quad \text{for } z\in S^1\, , \,\,  \Gamma(z) \in U(n)\, ,\,
\, \psi \in \mathbb{C}^n\, . 
\end{equation}
Quantum mechanics suggests that for all but finitely many points $z$,  equation \eqref{eqn:eigenvalue-eqn} has no
non-trivial solutions, and this will be taken as an axiom. 
Consequently, the solution to the problem is reduced to 
that of enumerating with multiplicity the
set of points $z \in S^1$ for which \eqref{eqn:eigenvalue-eqn} has a non-trivial solution.

From the perspective of algebraic topology, an advantage of equation \eqref{eqn:eigenvalue-eqn} is that it recasts the 
problem as a kind of intersection theory question in the unitary group $U(n)$: 
the  number of excitons is just the intersection number of the curve $\Gamma(S^1)$ with the  subspace $D_1U(n) \subset U(n)$
consisting of unitary matrices having non-trivial (+1)-eigenspace.
However, there is a catch: for $n \ge 2$, the space $D_1U(n)$ is not a manifold.
Nevertheless, $D_1U(n)$ is a stratified space and the intersection
problem turns out to be tractable.

\subsection{An intersection problem} The intersection problem alluded
to above is a special case of a more general problem, whose investigation
is the other goal of this paper.
Suppose $M$ is a closed smooth oriented manifold  of dimension $m$
equipped with a filtration by closed subspaces
\[
\emptyset= M_{-1} \subset M_0 \subset M_1\subset  \cdots \subset M_m = M
\]
defining a manifold stratified space \cite{Hughes}.
This means that the strata $M_{(j)} := M_j \setminus M_{j-1}$ are smooth manifolds
of dimension $j$, some which may be empty, and they satisfy the frontier condition: if $S$ and $T$ are strata such that $S\cap \bar T\ne \emptyset$, then
$S\subset \bar T$ (here $\bar T$ denotes the closure of $T$). Our convention
will be to omit the term $M_j$ from the filtration whenever
the corresponding stratum $M_{(j)}$ is empty.


For a fixed  degree $q$, set $Q := M_q$.  
On occasion, 
the singular homology group $H_q(Q;\Bbb Z)$ is found to be infinite cyclic.\footnote{This phenomenon is not as infrequent as one might initially suppose: 
For example, it occurs when
 the stratum $M_{(q)}$ is connected
 and the boundary homomorphism $H_q(M_q,M_{q-1};\Z ) \to
H_{q-1}(M_{q-1};\Z)$ has non-trivial kernel.}
When this happens,
an orientation of the manifold 
$M_{(q)}$ determines a preferred generator.
Denote the generator  by $\mu_q$.

Suppose that $f\: P \to M$ is a smooth map where $P$ is a closed oriented
manifold of dimension $p := m-q$. Let $[P]\in H_p(P;\Bbb Z)$ denote the fundamental class.
Then
the {\it homological intersection number} (or {\it global intersection index}) is defined as
\begin{equation} \label{eqn:int-number}
\alpha_f :=   \mu_q  \cdot f_\ast([P]) \in \Bbb Z \, ,
\end{equation}
where $\cdot\: H_q(M;\Bbb Z ) \otimes H_p(M;\Bbb Z) \to \Bbb Z$ is the intersection pairing.

Henceforth, assume
that the preimage 
\[
\Sigma := f^{-1}(Q)
\]
is a finite set of points.
\medskip

\begin{mainques} 
Is there a description of
 the homological intersection number \eqref{eqn:int-number}
given in terms of the restriction of $f$ to an arbitrarily small 
neighborhood of $\Sigma$?
\end{mainques} 

Stated another way, the question asks 
for a formula for the global invariant \eqref{eqn:int-number} 
in terms of local invariants coming from the differential topology 
of the map $f$ near the set
$\Sigma$.

In certain cases, the answer to the above question is well-known and classical. 
For example, if $Q\subset M$ happens to be a closed submanifold 
with fundamental class $\mu_q$ and $f$ happens to be transverse to $Q$, then
the homological intersection number
 is just the sum over the points of $\Sigma$ counted with sign.
However, in the case addressed in this paper, 
$Q$ is not a manifold and
the function $f$ is not assumed to satisfy any transversality.

In the case considered in the paper, $M$ is  $U(n)$, the unitary group
 of $n\times n$ matrices, which is
a compact smooth manifold of dimension $n^2$. The stratification of $U(n)$ is given
by the filtration
\begin{equation} \label{eqn:filtration-intro}
D_nU(n)  \subset \cdots  \subset D_1 U(n) \subset D_0U(n) = 
U(n)\, ,
\end{equation}
where $D_kU(n)$ is the set of matrices in $U(n)$ that fix
a vector subspace of $\Bbb C^n$ of complex dimension at least $k$. In this way of indexing the filtration, $M_{n^2-k^2} = D_{k}U(n)$. Note that
$D_nU(n)$ is the one point space given by the identity matrix.
Call \eqref{eqn:filtration-intro} the {\it multiplicity stratification}
of $U(n)$.

If we set $Q := D_1U(n)$, then
 $Q$ consists of the matrices in $U(n)$ having eigenvalue $+1$.
We will equip $U(n)$ with a CW decomposition whose $(n^2-1)$-skeleton 
is $Q$ and we will show that $H_{q}(Q;\Bbb Z)$ (for $q= n^2-1$)
is infinite cyclic with preferred generator $\mu$.
If $n \ge 2$, then $Q$ is not a manifold; see Example \ref{ex:U2} below. 

A main result (Theorem \ref{bigthm:indexthm}) of this paper is a kind
of {\it index theorem}. 
The crucial difference with the classical case
is that the sign of a point $\zeta\in \Sigma$
is now replaced by a {\it local index} $\iota_\zeta$, given by 
 the net change in the number of eigenvalues of $f(z)$ passing through $+1 \in U(1)$ 
as nearby points $z\in S^1$ pass through $\zeta$ in a counterclockwise direction
(see \S \ref{subsec:local} for details).

\begin{bigthm}[Index Theorem] \label{bigthm:indexthm} Assume
$f \: S^1 \to U(n)$ is smooth and the preimage 
$\Sigma =  f^{-1}(D_1U(n))$ is discrete. 
Then the global intersection index \eqref{eqn:int-number} of $f$ equals the
sum of the local indices, i.e.,  
\[
\alpha_f  = \sum_{\zeta \in \Sigma} \iota_\zeta\, .
\]
\end{bigthm}

In summary, Theorem \ref{bigthm:indexthm} provides a
description of the homological intersection number as a sum over local indices
in the particular case of $f\: S^1 \to U(n)$. The extent to which
a result of this kind applies in other instances remains
open.

\begin{figure}
  \begin{center}
    \includegraphics[width=0.40\textwidth]{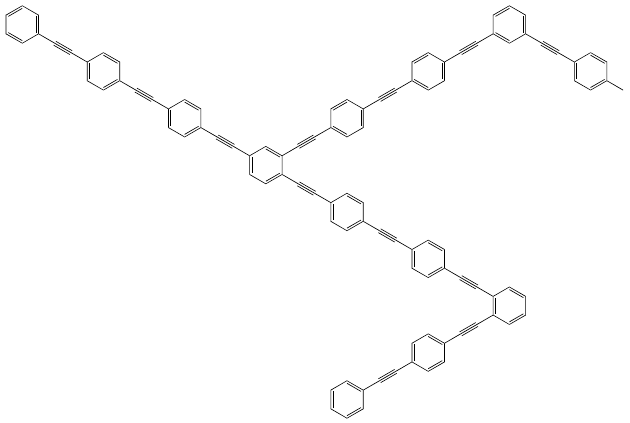}
    \includegraphics[width=0.40\textwidth]{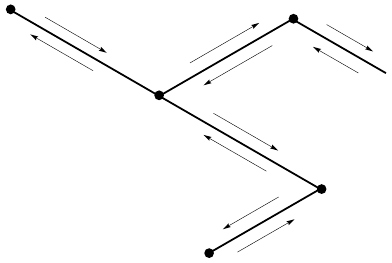}
  \end{center}
  \caption{Left: an example of the actual molecular geometry. Right:
  the metric graph on which it is based. The length of an edge
  is given by the number of repeat units it contains.}
 \label{side-by-side}
\end{figure}

\subsection{Exciton scattering} The molecules we investigate possess discrete,
1-dimensional translational symmetry which is broken only in a finite number
of places. The molecules are mathematically represented by
 simple metric graphs whose edge lengths
 are positive integers  (cf.\ Figure~\ref{side-by-side}). 

Let $Y$ be a connected, finite, simple graph. 
Let $Y_0$ denote the  set of vertices and $Y_1$ the  set of edges.
The graph $Y$ models the geometry of our molecule. Since $Y$ is simple, there is
at most one edge for each  pair of distinct vertices.
If $a,b\in Y_0$ are distinct vertices joined by an edge, 
we refer to the edge by $\{a,b\}$.

Associated with $Y$ is a directed graph $X$ called the {\it double}
of $Y$, whose set of vertices is $Y_0$ and where a directed edge
is given by a pair $(a,e) \in Y_0 \times Y_1$ such that $a$ is
incident to $e$. Each directed edge is uniquely 
represented by an ordered pair of distinct 
vertices $(a,b)$ and we
often specify $e$ as $ab$ (the arrows of Figure~\ref{side-by-side}
depict the directed edges of $X$). 

The lengths of the edges are given by a function 
\[
L \: Y_1 \to \{1,2,3,\dots\}\, .
\]
By slight abuse of notation, consider this as a function defined
on $X_1$ defined by $ab \mapsto L_{ab}$, with $L_{ba} = L(\{a,b\}) = L_{ab}$. 
From the viewpoint of chemical physics, the value $L_{ab}$ is the number of repeat units
of the edge $ab$ and corresponds to the number of periods of the discrete
symmetry on the edge.

If $a \in X_0 = Y_0$ is a vertex, let $X_1^{a,-}$ denote the set of incoming edges
at $a$, and similarly let $X_1^{a,+}$ be the set of outgoing edges at $a$.
There are canonical bijections $X_1^{a,\pm} \cong Y_1^a$, where $Y_1^a$ 
is the set of edges of $Y_1$ that meet $a$. The
bijection $X_1^{a,+} \cong Y_1^a$ is given by 
by $ab \mapsto \{a,b\}$ and the bijection  $X_1^{a,-} \cong Y_1^a$
is given by $ba\mapsto \{a,b\}$.


 

\begin{notation}
For a finite set $T$,
 let $\mathbb C[T]$ be the complex vector space with basis $T$. Let
 \[
 U(T)
 \] 
denote the unitary group of $\mathbb C[T]$ consisting of $\Bbb C$-linear transformations
$A\:\mathbb C[T]\to \mathbb C[T]$, such such that $AA^\ast = I$, where $A^\ast$
is the conjugate transpose of $A$.
 \end{notation}

The vertices of $X$ break the discrete symmetry, and hence, cause
the excitons to scatter (vertices are often referred
to simply as scattering centers for this reason). Associated to each vertex 
$a \in X_0$, and each real number $k\in\Bbb R$, there is  a 
{\it scattering matrix} 
\[
\tilde \Gamma^a(k)\: \mathbb C[ X_1^{a,-}] \to \mathbb C[X_1^{a,+}]\, 
\]
that describes the outgoing waves in terms of the incoming ones.\footnote{The real number $k$ is known as {\it quasimomentum}; it appears as a consequence of translational invariance.}

The scattering matrix is invertible and
the function $k \mapsto \tilde \Gamma^{a}(k)$
 is $(2\pi)$-periodic. In view of the canonical bijections $X_1^{a,\pm} \cong Y_1^a$,
the scattering matrix at $k \in \Bbb R$ can equally be regarded an invertible
self-map
\begin{equation} \label{eqn:gamma-a-k}
\tilde \Gamma^a(k)\: \mathbb C[ Y_1^{a}] \to \mathbb C[ Y_1^{a}]\, .
\end{equation}


\begin{hypo} \label{hypo:analytic} 
  For each $a\in X_0$ and $k \in \Bbb R$, the operator of
\eqref{eqn:gamma-a-k} is
unitary. Furthermore, 
the map $\tilde \Gamma^a\: \R \to U(Y_1^a)$ given by $k \mapsto \tilde \Gamma^a(k)$
is real analytic. 
In particular, the map
\[
 \Bbb R \times \Bbb C[ Y_1^{a}] \to \Bbb C[ Y_1^{a}]\, ,
\]
given by $(k,\psi) \mapsto \tilde \Gamma^a(k)\psi$, is  real analytic.
\end{hypo}

\noindent The hypothesis, which we take as an axiom, is
suggested by quantum mechanics~\cite{LandauLifshitz:QM}. It has been
substantiated in a variety of cases. 
Periodicity implies that $\tilde \Gamma^a$  factors as a map through the
circle. I.e.,
\[
\Gamma^a(e^{ik}) := \tilde \Gamma^a(k)  \, ,
\]
where $\Gamma^a\: S^1 \to U(Y_1^{a})$ is also real analytic.

Excitons  satisfy {\it time-reversal
symmetry (Kramers' symmetry)}:
\begin{equation} \label{eqn:time_reversal}
\tilde \Gamma^a(k)^{\ast} = \tilde \Gamma^a(-k)\, .
\end{equation}
In terms of $\Gamma^a$, this equation translates to 
$(\Gamma^a)^{\ast}(z) = \Gamma^a(\bar z)$, which means that
$\Gamma^a$ is $\Bbb Z/2\Bbb Z$-equivariant with respect to conjugation
on the circle and inversion on $U(Y_1^a)$. 
  
The scattering matrices are computed using quantum chemistry methods, and we
treat these as known. 

\subsection{The Exciton Scattering Equations}
For a vector
$\psi \in \mathbb{C}[X_1]$ 
and a directed edge $ab \in X_1$, let
$\psi_{ab}\in \mathbb C$ be the component of $\psi$ along $ab$. 
Hence,
$\psi = \sum \psi_{ab}ab$.  


Exciton scattering theory implies 
the following two equations,  which are taken here as axioms \cite[eqns.~(3),(4)]{Wu:ESI}:
\begin{equation} 
  \label{eq:ESold}
\begin{aligned}
  \psi_{ba} &= e^{ikL_{ab}} \psi_{ab}\, ,\\
  \psi_{ba} & = \sum_{ac \in X_1} 
     \Gamma^a_{ba,ac}(e^{ik}) \psi_{ac}\, ,
\end{aligned}
\end{equation}
where in the indexing for the sum the vertex $a$ is fixed. 
There is such a pair of equations for every edge in $X$.  
The entire system is referred to as the {\it ES equations.}

The matrix entries  $\Gamma^{a}_{ba,ac}$ can be
interpreted as the amplitude of the wave function $\psi$ on edge $ba$ as the result of
scattering a Gaussian plane wave from edge $ac$ onto vertex $a$~\cite{LandauLifshitz:QM}.

\begin{rem}
The first equation in ~(\ref{eq:ESold}) expresses the fact that when an
exciton wave propagates along a linear segment, it acquires a phase change. The
phase change depends on both the length of the segment and 
$e^{ik}$.  The second equation in  ~(\ref{eq:ESold}) connects the amplitude of
outgoing waves to those of the incoming.  
\end{rem} 

\begin{rem} The ES equations
of \cite[eqns.~(3),(4)]{Wu:ESI} are written in a slightly different notation
than the equations of \eqref{eq:ESold}. The notation in the former case
is intended for an audience of chemists and physicists, whereas \eqref{eq:ESold}
 is geared toward mathematicians.
\end{rem}

To bring algebraic topology into the picture, the equations 
need to be put into a compact form.
Let 
$\tilde \Gamma_0$ be the composition
\[
\Bbb R @> \Delta >> \prod_{a\in X_0} \Bbb R @>\prod_a \tilde \Gamma^a >>  \prod_{a\in X_0}  U(Y_1^a) @>\oplus_a >> U(X_1)\, ,
\]
where $\Delta$ is the diagonal and  the 
map $\oplus_a$ is defined by block sum of matrices (here, we used
 the canonical bijection $X_1 \cong \amalg_a Y_1^a$).

Choose once and for all a linear ordering on $Y_0$.
Then the left-lexicographical ordering on ordered pairs
identifies $X_1$ with the standard ordered basis for $\Bbb C^n$ where
$n = |X_1|$. In this way, $\tilde \Gamma_0$ becomes a $(2\pi)$-periodic real analytic map
$\Bbb R \to U(n)$.

Passing to the circle, one obtains a map $\Gamma_0\: S^1 \to U(n)$ defined by
\[
\Gamma_0(z) := \tilde \Gamma_0(k)\, , \quad \text{for }  z = e^{ik}\, .
\]
Let
\[
\hat{L}: \mathbb C[X_1] \ra \mathbb C[X_1] 
\]
be the length rescaling operator given by 
 $\hat L(ab) = L_{ab}ab$. 
Finally, define a map  $\Gamma\: S^1 \to  U(n)$ by 
\[
 \Gamma(z) := e^{ik\hat{L}}\cdot \Gamma_0(e^{ik})\, , \quad z = e^{ik} \, . 
\]
One verifies by  direct calculation that the
 full set of exciton scattering equations is equivalent to
the single parametrized eigenvector equation
\begin{equation}\label{eq:ES-circle}
 \Gamma(z) \psi = \psi\, , \quad  z\in S^1\, ,
\end{equation}
where $\Gamma\: S^1 \to U(n)$ is real analytic,
and $0 \ne \psi \in \Bbb C^n$ is a vector.
Henceforth, we resort to this version of the ES equations throughout the paper.

\subsection{Enumeration of excitons}
The analyticity of $\Gamma$ implies that equation ~\eqref{eq:ES-circle} can 
hold only for finitely many points $z \in S^1$.  Specifically, let $S^{2n-1}
\subset \mathbb C^n$ be the unit sphere. By Hypothesis \ref{hypo:analytic}, the
map
\[
F\: S^1 \times S^{2n-1} \to \mathbb C^n
\]
given by $F(e^{ik},v) = \Gamma(e^{ik})v - v$ is analytic. By the principle of permanence and the compactness
of $S^1 \times S^{2n-1}$, the set of zeros $F^{-1}(0) \subset S^1 \times S^{2n-1}$ is  finite. The image of the first
factor projection $p_1\: F^{-1}(0)\to S^1$ is then a finite set of points
\[
e^{ik_1},e^{ik_1},\dots,e^{ik_\ell}\in S^1\, ,\quad  0 \le k_1 < k_2 < \cdots < k_\ell < 2\pi\, .
\]
This is just the set of points  $z\in S^1$ where $\Gamma(z)\psi = \psi$ has a non-trivial solution.
Let $z_j = e^{ik_j}$.

For $1\le j\le \ell$, let $m_j$ denote the dimension of the $(+1)$-eigenspace of 
$\Gamma(z_j)$.
Then $m_j$ is the number of linearly independent solutions of the ES
equations at $z_j$. Recall the problem is to count
the total number of 
excitons weighted with their multiplicities, i.e., $m = \sum m_j$. The following
result, which is an application of Theorem \ref{bigthm:indexthm},  provides a lower bound
for $m$.

\begin{bigthm}[Exciton Lower Bound]
  \label{bigthm:lower-bound}
  Let $m$ denote the total number of solutions of the ES equations \eqref{eq:ESold}, weighted
  with their multiplicities. Then
  \[ 
    m \,\,\geq \,\, \sum_{ab \in X_1} L_{ab} + \sum_{a \in X_0}w(\Gamma^a)\,,
  \] 
  where $\Gamma^a$ is the scattering matrix at the vertex $a$ and
  $w(\Gamma^a)$ is the degree of the map $\det \circ \, \Gamma^a : S^1 \ra U(1)$.
\end{bigthm}
 
There is evidence that the inequality 
of Theorem \ref{bigthm:lower-bound} is sharp. 
In fact, for molecules with large segment lengths, one has the following result.

\begin{bigadd}[The Long-Arm Limit]
  \label{bigthm:long-arm} Up to second order approximation 
  and for sufficiently large segment lengths, the bound of Theorem \ref{bigthm:lower-bound} 
  becomes an equality.
\end{bigadd}

\begin{rem} ``Up to second order approximation''
refers to exchanging the one-parameter family of operators $\Gamma(z)$ 
with its associated 1-jet at the solution set.
Empirical evidence from physics indicates that the addendum is 
valid even without resorting to the approximation.
\end{rem}

\subsection{Historical background} 
The theory of excitons was launched by Frenkel, Peierls, and Wannier in the
1930s  \cite{Frenkel, Peierls, Wannier}.  In its early years the
subject grew slowly and sporadically. The theory blossomed in the 1950s as a result of 
an increase in technology as well as 
the drive for a better understanding of the optical properties of solids  \cite{Overhauser}.  
For an early survey, see the introduction to \cite{Knox}.

The sort of excitations considered here are produced when an organic structure,
like chlorophyll, absorbs light of a certain wavelength.  The absorbed photon 
excites an electron into a higher energy state and the excited electron
can form a bound state with the ``hole'' it
leaves behind. The pair consisting of the excited electron and the hole is
known as an {\em exciton}. 
Finding a method to accurately enumerate excitons in an organic molecule 
has  consequences for engineering. For example, it plays an important role
in the design of more efficient organic solar cells.

The exciton scattering (ES) approach, introduced in \cite{Wu:ES} (see
also \cite{Wu:ESI}), was a dramatic improvement
over previous attempts to study electronic excitations in branched conjugated 
molecules for a number of reasons (for an extensive bibliography, 
see \cite{Wu:ESI, Wu:ES2, Wu:ES3}). The ES approach has 
shown itself to be more powerful than
traditional quantum chemistry methods,
such as Time Dependent-Density Functional Theory (TDDFT):\footnote{TDDFT and all
other quantum chemistry methods consist of  
approximation  schemes for solving the Schr\"odinger equation for a many-electron system.}
the computational time with the ES approach is on the order of seconds, whereas the
same computations with traditional methods are on the order of hours or days.

In earlier papers, the ES
approach was only applied to molecules with perfect symmetries, e.g. an `X' or
`Y' joint with the same number of repeat units on each arm~\cite{Li:symmetric}.
For example, the molecule displayed in Figure~\ref{side-by-side} fails to have
perfect symmetry and is therefore not capable of study using previous techniques.
The symmetric case was done by studying the equations ~\eqref{eq:ESold} using the
full symmetry group of the molecule (and the representation theory thereof).
Rewriting the ES equations in the form \eqref{eq:ES-circle} above
 allows for the analysis of
electronic excitations of general branched, conjugated molecules done in this
paper, as well as in \cite{Catanzaro:CountingExc}, regardless of symmetry.
Stated differently, this paper is the result of analyzing the ES equations for
generic molecules, in a mathematically rigorous way.

Theorem \ref{bigthm:lower-bound} exhibits an estimate for $m$, the number of solutions to the ES equations
weighted by their multiplicities. 
Chemical physicists are interested instead in
a number, denoted by $N$, which is the number 
  of excitations that lie in a fixed energy range, called the {\em
exciton band}.  The two numbers are related by the formula
\begin{equation} \label{eqn:formula}
  2N  \,\, = \,\, {m + d_0 + d_\pi}\, ,
  \end{equation}
where $d_k = d_k^+ - d_k^-$ and $d_k^\pm$ is the dimension of 
the $(\pm 1)$-eigenspace of  $\tilde\Gamma(k)$ for $k = 0,\pi$.
(cf.\ \cite[eqn.~(12)]{Catanzaro:CountingExc}). 
The formula \eqref{eqn:formula} is a reflection 
of the observation that the solutions of the ES equations
with $k \ne 0,\pi$ come in pairs, by time-reversal symmetry \eqref{eqn:time_reversal}. 
The appearance of the terms $d_k$ account for the identity
 $\tilde \Gamma(k)^2 = I$, when $k = 0, \pi$, which is again a consequence of time-reversal symmetry.
The fundamental domain $[-\pi,\pi]$ is known as the {\it Brillouin zone.}

\subsection{Remark on the exposition}
We have written the paper so that it is accessible to readers
with a modest background in algebraic and differential topology.  
The material from chemical physics---especially the technical jargon---has been kept to a bare minimum and is usually relegated to remarks or footnotes in an effort to make the paper accessible to a broad mathematical audience. 

\begin{out}
The CW decomposition of $U(n)$ is
described in \S\ref{subsec:steenrod}. The technical results required for the proof
of Theorem \ref{bigthm:indexthm} appear in \S\ref{sec:indexthm}. 
 An exposition of the proof
Theorem \ref{bigthm:indexthm} appears in \S\ref{sec:proof-index-thm}. 
The proof of Theorem \ref{bigthm:lower-bound} 
appears in \S\ref{sec:lower-bound} and the proof of Addendum 
\ref{bigthm:long-arm} appears in \S\ref{sec:long-arm}.
\end{out}

\begin{acks} 
  The authors thank Tian Shi for helpful discussions and
  for assistance with the figures.  The authors also thank
  a diligent referee whose numerous suggestions greatly improved
  the readability of the paper.    
  The third author is indebted to Greg Arone
  for explaining to him the argument in the general case that appears in the proof of Proposition~\ref{prop:mult_in_hom}.
    
  This material is based upon work supported by the National
Science Foundation Grant  CHE-1111350
and the Simons Foundation Collaboration Grant  317496.
\end{acks}



\section{The stratification of $U(n)$} \label{subsec:steenrod}

We describe the CW decomposition
of $U(n)$ appearing in Steenrod's book \cite{Steenrod:CohomOps}. 
Steenrod credits the CW structure to C.~Miller \cite{C-Miller} and Yokota \cite{Yokota}. We then relate the CW structure to the multiplicity stratification.

\subsection{The CW structure}
Following~\cite[ch.~IV]{Steenrod:CohomOps}, let $E^{2(n-1)} \subset S^{2n-1} \subset \mathbb C^n$ be the subspace
consisting of vectors whose final coordinates are real and nonnegative, where
we are using the standard ordered basis on $\mathbb{C}^n$.
The circle $S^1\subset \Bbb C$ acts freely on $S^{2n-1}$ by $\nu\cdot (z_1,\dots,z_n) = 
(\nu z_1,\dots,\nu z_n)$.
  Let $Q_{n}$
be the quotient space of $S^1 \times S^{2n-1}$, where $(\lambda, x) \sim
(\lambda,\nu\cdot x)$ for $\nu \in S^1$, and $(1,x) \sim (1,y)$ for all
$x,y\in S^{2n-1}$. The space $Q_n$ is homeomorphic to the reduced suspension
$\Sigma (\CP^{n-1}_+)$, where $\CP^{n-1}$ is the 
projective space of $\Bbb C^n$ and  $X_+$ denotes the effect of 
adding a disjoint basepoint to a space $X$. A point of $Q_n$ will be written
as an equivalence class $[\lambda,v]$, where $\lambda \in S^1\subset \Bbb C$ and $v \in S^{2n-1}\subset \Bbb C^n$.

For each $j \le n$, 
there is a natural map 
\[
  E^{2j-1} := E^{2(j-1)} \times [0,1] \ra S^{2j-1} \times S^1 \ra Q_j \, ,
\]
where the second displayed map is the projection onto orbits,
and the first is the product of the inclusion $E^{2(j-1)} \subset S^{2j-1}$
with the
map $[0,1] \ra S^1$ given by $t\mapsto e^{2\pi it}$.
The displayed map defines a relative homeomorphism 
\begin{equation} \label{eqn:char}
(E^{2j-1},S^{2j-2}) \to  (Q_j, Q_{j-1})\, .
\end{equation}
The maps \eqref{eqn:char} define
a CW decomposition of $Q_n$ which is just the suspension of the 
usual CW structure of $\CP^{n-1}_+$.

An embedding of $Q_n$ into $U(n)$ is given by sending an 
equivalence class $[\lambda,v] \in Q_n$ to the transformation which fixes the 
orthogonal compliment of the line $\Bbb Cv \in \mathbb C^n$ and multiplies any point on
 $\Bbb Cv$ by $\lambda$ (by convention, $Q_0$ maps to the identity transformation).  
This gives rise to embeddings $Q_j \ra U(n)$ for $j \leq n$ using the
inclusion $U(j) \ra U(n)$, given by including the first $j$ coordinates.
Multiplying the cells leads to {\em normal cells}: let $\bar n$ be
the ordered set  $\{1 < 2 < \cdots < n\}$. Then for  
any ordered subset 
\[
  I = \{i_1 < i_2 < \cdots < i_r\} \subset \bar n \, ,
\]
there is a (characteristic) map
\begin{equation} \label{eqn:normal-cells}
  E^I := E^{2i_1-1} \times E^{2i_2 -1} \times \cdots \times E^{2i_r-1} \to
  Q_{i_1} \times Q_{i_2} \times \cdots \times Q_{i_r} \to U(n)
\end{equation}
where the first map of \eqref{eqn:normal-cells} is the product of the 
 maps \eqref{eqn:char} and 
the second map is given by multiplication. The normal cells (together with 
the 0-cell corresponding to the identity) determine a CW decomposition
of $U(n)$~\cite[ch.~IV, thm.~2.1]{Steenrod:CohomOps}.

\begin{ex} \label{ex:U2} 
  Let $n = 2$. The above equips $U(2)$ with a CW structure with a single cell in dimensions $0,1,3$ and $4$.  
 The $3$-skeleton is $D_1U(2) = \Sigma (\CP^1_+)$ which may be thought of as the quotient space of $S^3$ obtained by identifying a single pair of antipodal points,
  as shown in Figure~\ref{fig:rendering} below.

There is another CW structure on $U(2)$ which uses the well-known homeomorphism
$U(2) \cong S^1 \times S^3$. This CW structure  also has a single cell in
dimensions $0,1,3$ and $4$, and the $3$-skeleton in this case is $S^1 \vee
S^3$. Although $\Sigma (\CP^1_+)$ is homotopy equivalent to $S^1 \vee S^3$, the
two spaces are not homeomorphic. Hence the two CW structures are distinct
(in fact, $SU(2)$ fails to be a subcomplex of the
first CW structure).

Blurring the distinction 
between the two CW structures in this paper invariably leads 
to unrecoverable errors. 
In particular,  the proof of $n=2$ case of Proposition \ref{prop:mult_in_hom} 
goes wildly wrong were we to use $S^1 \vee S^3$ in place of $D_1U(2)$.
\end{ex}

\subsection{The multiplicity stratification} 
We now return to the filtration
\eqref{eqn:filtration-intro}.  By inspection, the image of the normal cell \eqref{eqn:normal-cells}
is contained in $D_{n-r}U(n)$. If we define
\[
D_{(k)}U(n) :=  D_{k}U(n) \setminus D_{k-1}U(n)\, ,
\]
then  there is a canonical bijection of sets
\begin{equation} \label{eqn:decomp}
D_{(k)}U(n)  \,\, \cong \,\, \coprod_{|I|=n-k} \mathring{E}^I\, ,
\end{equation}
where each summand on the right maps homeomorphically onto its image on the left.

However, the topology on the right which makes
the identification a homeomorphism isn't the disjoint union topology. 
In fact, \cite[th.~A]{H-Miller} shows that 
 $D_{(k)}U(n)$ is diffeomorphic to the total space
of a real vector bundle of rank $(n-k)^2$ over the  Grassmannian $G_k(\Bbb C^n)$ of complex $k$-planes in $\Bbb C^n$ (see \cite[th.~A]{H-Miller}). 
It follows that $D_{(k)}U(n)$ is a connected smooth manifold of dimension $n^2-k^2$. It is trivial to check that the frontier condition is satisfied,
so the multiplicity stratification defines a manifold stratified space.

In particular, \eqref{eqn:decomp} implies the following result.

\begin{lem} The subspace
 $D_1U(n)\subset U(n)$ coincides with $(n^2{-}1)$-skeleton of the CW decomposition
 of $U(n)$ defined by the normal cells, and 
\begin{equation} \label{eqn:top-cell}
E^3 \times E^5 \times \cdots \times E^{2n-1} \to D_1U(n)
\end{equation}
is the unique  normal cell of dimension $n^2-1$.
\end{lem}

  We conclude this section with some remarks.
  
\begin{rems} (1). The above CW decomposition of $U(n)$ has the property
that there are cells in each dimension
$k \le n^2$ except when $k =2,n^2-2$.
\medskip

\noindent (2). Even though $D_1U(n)$ is the $(n^2-1)$-skeleton, \eqref{eqn:decomp} shows that
the multiplicity stratification  isn't  skeletal.
\medskip

\noindent (3). For $I \subset \bar n$, the 
operation $I \mapsto \bar n \setminus I$ determines
an involution on the set of normal cells and gives rise to Poincar\'e duality 
on $U(n)$.
\end{rems}




\section{Intersection Theory in $U(n)$} \label{sec:indexthm} 
\subsection{The global index}
The CW decomposition of $U(n)$ given in \S\ref{subsec:steenrod}
has one top cell and its $(n^2-1)$-skeleton is $D_1U(n)$.
 The subcomplex $D_1U(n)$ also has one top cell given by  \eqref{eqn:top-cell}.
 The latter defines a generator $\mu \in H_{n^2-1}(D_1U(n))$
 (with integer coefficients being understood; the top cell of $D_1U(n)$
 defines a cycle since there are no cells in dimension $n^2-2$). 
 Alternatively, the generator
can be obtained by means of Poincar\'e duality: 
 the fundamental class $[U(n)] \in H_{n^2}(U(n))$ defined by 
 the top cell of $U(n)$ gives
an isomorphism
\[
H_{n^2-1}(D_1U(n)) @> j_\ast >\cong > H_{n^2-1}(U(n)) @<\cap [U(n)] <\cong < 
H^1(U(n)) \cong  \Z \, ,
\]
where the first displayed map is induced by the inclusion
 $j:D_1U(n) \ra U(n)$ and the second displayed map is
Poincar\'e duality. The generator of $H^1(U(n))$ comes from the determinant map
$\det\: U(n) \to S^1$. The map $j_\ast$ is an isomorphism by
the homology long exact sequence of the pair $(U(n),D_1U(n))$.
Let $[S^1] \in H_1(S^1)$ be the
fundamental class.

\begin{defn}
  The {\em global intersection index} of a map $f\: S^1 \to U(n)$ 
  is the homological intersection number
  \[
     \alpha_f  =  j_*(\mu) \cdot f_*([S^1])   \in H_0(U(n);\Bbb Z) \cong  \Bbb Z.
  \]
\end{defn}

\begin{defn} The {\em intersection set} is the fiber product
  \[
    \mathfrak I_f :=  \{ (A,z) \in D_1U(n) \times S^1 |\, A = f(z)  \}.
  \]
\end{defn}

\begin{hypo} The set  $\mathfrak I_f$ is discrete, i.e., $\Sigma = f^{-1}(D_1U(n))$
is a finite set.
\end{hypo}

It will be convenient to express the global intersection index in terms of orientation
classes~\cite[p.~294]{Spanier:AlgTop}.

\begin{notation}  
For a subset $A \subset X$, let 
\[
(X\,|\, A) \, :=\,  (X,X\setminus A)\, .
\] 
Note the case $A = X$ 
is the pair $(X,\emptyset)$.
The case  $A = p$ is a point gives rise to 
the {\it local (singular) homology} at $p$, i.e.,  
 $H_\ast(X\,|\, p) := H_\ast(X,X\setminus p)$. \end{notation}

\begin{defn} Let $X$ be a smooth $m$-manifold and let $\Delta \subset X \times X$ be the diagonal. An {\em orientation} for $X$ is a (singular cohomology) class
  \[
x \in H^{m}\left( X\times X\, | \, \Delta \right)
\]
  such that for all $p \in X$, the class $i_p^*(x)$
  generates $H^{m}(X\, |\, p) \cong \Bbb Z$, where  $i_p$ is the inclusion
  \[
 (X \times p\, | \, p \times p) \subset (X\times X \, |\, \Delta ) \, .
    \]
\end{defn}

\begin{rem} The terminology is not universal. An orientation in the above 
sense is equivalent to
choosing a Thom class for the tangent bundle $TX$.
In \cite[p.~123]{Milnor-Stasheff}, $x$ is  called a 
{\it fundamental cohomology class}. 
\end{rem}

A choice of orientation class 
\[
x\in H^{n^2}(U(n)\times U(n) \, | \, \Delta)\, 
\]
is equivalent to a choice of fundamental class for $H_{n^2}(U(n))$. 
Henceforth, $x$ is chosen to correspond to the class $[U(n)]$ defined above.

Let $b\: (U(n) \times U(n)\, |\, U(n)\times U(n)) \to (U(n) \times U(n)\,  |\, \Delta)$ be the evident map.

\begin{lem}\label{lem:globalint}
  The global intersection index
 coincides with the Kronecker index, i.e., 
 $\alpha_f =  \langle x, b_\ast(j\times f)_\ast(\mu \times [S^1])\rangle$.
 \end{lem}

\begin{proof} This follows immediately from the standard description of
the intersection number as a slant product 
(cf.\ ~\cite[chap.~6]{Spanier:AlgTop}, \cite[\S4]{McCrory}).
\end{proof}

\subsection{The local multiplicity} \label{subsec:multiplicity}
The local multiplicity, defined below, counts the dimension of the solution space
of the equation 
\begin{equation} \label{eqn:solutions}
f(z)\psi =\psi \, ,  \qquad \psi \in \Bbb C^n \, ,
\end{equation}
at a given $z\in S^1$.  

\begin{defn}\label{defn:multiplicity}
Let $p = (A,z) \in  \mathfrak I_f$ be any point.
  The {\em local multiplicity} of $p$, denoted $m_p$,
  is the dimension of the 
  (+1)-eigenspace of $f(z) = A $. In particular, $m_p \ge 1$.
  
  The {\it local multiplicity} of $f$ is the number
  \[
  m = \sum_{p} m_p\, ,
  \]
  where the sum is indexed over all points $p \in  \mathfrak I_f$.
\end{defn}

The local multiplicity of $f$ enumerates the entire set
of solutions to  equation \eqref{eqn:solutions}, where each solution $p$
is counted with multiplicity $m_p$.

\subsection{The local intersection index} \label{subsec:local}
 For $p = (A,z) \in  \mathfrak I_f $,
let $k_p \in [0,2\pi)$ be the unique point such that $z = e^{ik_p}$. 
By hypothesis, the 
non-trivial solutions to the equation $f(z)\psi = \psi$ occur at isolated 
points $z \in S^1$. 

Choose a small arc $J$ containing $+1\in S^1$ such that
the eigenvalues of $f(z)$ other than $+1$ lie outside $J$.
Consider an interval $I \subset \R$ centered at $k_p$. If $k \in I$, then the number of
eigenvalues of $f(e^{ik})$ which lie in $J$ is constant in $k$
provided that $I$ is chosen sufficiently small.  If $k \in I \setminus
\{k_p\}$ then the number of  eigenvalues of $f(e^{ik})$ lying in $J$ and having positive
imaginary part is a constant function on each component.  

\begin{defn} \label{defn:iota-pm} 
  With respect to the above assumptions,  
  let $\iota_p^-$ be the
number of eigenvalues lying in $J$ having positive imaginary part taken at $k < k_p$ with $k \in I$.
Similarly,  
let $\iota_p^+$ be the
number of eigenvalues lying in $J$ having positive imaginary part taken at some $k > k_p$ with $k \in I$.

The  {\it local intersection index} at $p \in \mathfrak I_{f}$ is
  the integer \[ \iota_p := \iota_p^+ - \iota_p^-\, .  \] The {\it local intersection
  index}  of $f\: S^1 \to U(n)$ 
  is the sum \[ q_f := \sum_{p \in \mathfrak I_{f}} \iota_p\, .  \]
\end{defn}
Intuitively,  $\iota_p$ is the number of eigenvalues passing through 
$1 \in U(1) \cong S^1$ counterclockwise, as $k$ passes through
$k_p$ from left to right. 

The definition trivially
implies the inequalities
\begin{align} \label{eqn:m>p} 
\begin{split}
m_p &\ge \iota_p\, ,\\
m  &\ge q_f\, 
\end{split}
\end{align}
(cf.\ Definition \ref{defn:multiplicity}).

The following result
 relates the local homology of a point
in $D_1U(n)$ to the rank of its $(+1)$-eigenspace. 

\begin{prop}\label{prop:mult_in_hom} Let $g\in D_1U(n)$ 
be a matrix whose $(+1)$-eigenvalue has multiplicity $j$. Then there is 
an isomorphism
\[
H_{n^2-1} (D_1 U(n) \, |\,  g ) \cong \Z^{j}\, .
\]
\end{prop}

\begin{proof}
Assume that $n \ge 2$, as the $n = 1$ case is trivial. We consider two cases.
\medskip

\noindent {\it Case 1:} $g = I$ is the identity. In this case we need to 
produce an isomorphism $H_{n^2-1} (D_1 U(n) \, |\,  I ) \cong \Z^{n}$.

Recall that the Lie algebra $\mathfrak u(n)$ is given by the skew
Hermitian matrices (i.e. $n \times n$ matrices $A$ satisfying $A^\ast = -A$;
recall that the eigenvalues of such an $A$ are pure imaginary).
  Fix $\epsilon > 0$ and take a standard neighborhood 
  $U_\epsilon := \{A \in \mathfrak
  u(n) \, |\,  ||A|| \le \epsilon \}$ using the operator norm. 
 Then $||A||$ coincides with $\max_d |d|$, where $di$ ranges over the set of eigenvalues of $A$.  Under the 
  exponential map, the elements of $U_\epsilon$ that map to $D_1U(n)$ are given by the set 
  \[
  V_\epsilon := U_\epsilon \cap \exp^{-1}(D_1U(n))
  \] 
which is just the set of singular skew Hermitian matrices having norm at most $\epsilon$. 
  As $\exp$ is a local diffeomorphism for $\epsilon$ sufficiently small, it follows that the local homology $D_1 U(n)$ at $g= I$
  coincides with the local homology of $V_\epsilon$ at the zero vector $0$. 
Hence, to complete Case 1, it is sufficient to identify the latter.
  
Observe that
  $V_\epsilon$ is contractible (a contraction is given by 
  the formula $(A,t) \mapsto (1-t)A$ for $t\in [0,1]$ and $A\in V_\epsilon$). Furthermore, $V_\epsilon \setminus 0$ admits a deformation retraction to
 $\partial V_\epsilon \subset \partial U_\epsilon \cong S^{n^2-1}$, where
 $\partial V_\epsilon$ is the set of matrices of $V_\epsilon$ 
 with all non-zero eigenvalues
of absolute value equal to $\epsilon$. To construct the deformation retraction,
move $A$ without changing its eigenspaces by deforming each non-zero 
eigenvalue $\lambda = di$
using the formula
\[
\lambda_t := \begin{cases} 
 (t\epsilon + (1-t)d)i \quad & \text{ if } d > 0\, ,\\
 (-t\epsilon + (1-t)d)i & \text{ if } d < 0\, ,
\end{cases}
\]
for $t\in [0,1]$.
The above shows that the local homology of $V_\epsilon$ coincides with 
$H_{n^2-2}(\partial V_\epsilon)$. Set $X := S^{n^2-1} \setminus \partial V_{\epsilon}$.

By Alexander duality, we infer
\begin{equation}\label{eqn:rank}
H_{n^2-2}(\partial V_\epsilon) \cong 
\tilde H^0(X)\, .
\end{equation}    
The  unreduced cohomology group $H^0(X)$ is free abelian 
  on the number of path components of $X$. 
  For $A\in X$, define
  $s_+(A)$ to be the number of eigenvalues $di$ of $A$ such that $d > 0$.
  We claim that two matrices $A$ and $B$ lie in the same component of $X$ 
  if and only
  if $s_+(A) = s_+(B)$. 
 The latter assertion follows from the fact that we can deform
  $A$ and $B$ to diagonal matrices inside $X$ without changing their eigenvalues (the deformation
  is given by conjugating with paths of unitary matrices to the identity).  This enables us to assume that $A$ and $B$ are diagonal.
However, diagonal matrices $A$ and $B$ lie in the same component of $X$
if and only if $s_+(A) = s_+(B)$.
It follows that the function $s_+$ induces a bijection
 \[
\pi_0(X) \cong \{0,1,2,\dots, n\}\, .
\]
Hence, $|\pi_0(X)| = n+1$, 
and we infer that the group \eqref{eqn:rank} is free abelian
 of rank $n$. This completes Case 1.
 \medskip
 
 \noindent {\it Case 2:} The general case. It is  enough the consider the case
 when $g$ is the matrix
\[
\begin{bmatrix}
I_j & 0 \\
0 & D 
\end{bmatrix}
\] 
where $I_j \in U(j)$ is the identity matrix and
 $D \in U(n-j)$ is a diagonal matrix such that no entry is equal to 1.
To justify this, note that for any $h \in D_1U(n)$ there is 
a $P\in U(n)$ such that $PhP^\ast = g$, where $g$ is as above (the number $j$
is the rank of the $(+1)$-eigenspace of $h$).
Then conjugation by $P$  transfers
 between the neighborhoods of $h$ and
 the neighborhoods of $g$.
 
To complete the proof of the proposition, 
it will suffice to construct  an isomorphism
\begin{equation}\label{eqn:reduction}
H_{n^2-1}(D_1U(n)|g) \cong H_{j^2-1}(D_1U(j)|I_j) \, , 
\end{equation}
since by Case (1), there is an isomorphism $H_{j^2-1}(D_1U(j)|I_j) \cong \Bbb Z^j$.

Since $U(j) \subset U(n)$ is a smooth submanifold of codimension $n^2-j^2$,
the point  $g \in U(n)$ has a  neighborhood $U$ diffeomorphic to
\[
E\times \Bbb R^{n^2-j^2} \, ,
\]
where $E \cong \Bbb R^{j^2}$ is a neighborhood of $I_j \in U(j)$. Choose the diffeomorphism $E\times \Bbb R^{n^2-j^2} \to U$ so that it restricts to the identity on $E\times 0$. In what follows, we use the diffeomorphism to identify $U$ with $E\times \Bbb R^{n^2-j^2}$.

 For $\epsilon > 0$,
let $E_\epsilon\subset E$ correspond to the ball of radius $\epsilon$ in $\Bbb R^{j^2}$. Let
$B_\epsilon \subset \Bbb R^{j^2}$ be ball of radius $\epsilon$ in 
$\Bbb R^{n^2-j^2}$. Set
\[
U_\epsilon := E_\epsilon \times B_\epsilon\, , \quad V_\epsilon := U_\epsilon \cap D_1U(n)\, .
\]
Let $h \in U_\epsilon$ be any element.  If $\epsilon$ is sufficiently small, then the eigenspaces of $h$ bifurcate into two types: 
those which  are close to $\Bbb C^j\times 0 \subset \Bbb C^n$ and those which are close to $0 \times \Bbb C^{n-j}$.  
Those of the former type have eigenvalues lying on a small open arc
containing $+1$ on the unit circle $S^1$, 
while those of the latter type lie on the complement of the arc. 
If $h$ also lies in $D_1U(n)$, then at least one of the eigenvalues of the first type must be equal to +1. In other words
\[
V_\epsilon  = (E_\epsilon \cap D_1U(j)) \times B_\epsilon \, .
\]
Apply the relative K\"unneth theorem to the above to obtain an isomorphism
\[
H_{n^2-1}(V_\epsilon|g) \cong 
H_{j^2-1}(E_\epsilon\cap D_1U(j)|I_j) \, .
\] 
By excision, the left hand side coincides with $H_{n^2-1}(D_1U(n)|g)$ and
the right hand side coincides with $H_{j^2-1}(D_1U(j)|I_j)$, establishing the
isomorphism \eqref{eqn:reduction}, and completing Case 2.
\end{proof}

\begin{rem} \label{rem:plausible} 
The proof of Proposition~\ref{prop:mult_in_hom} delivers a bit more: it shows
that the multiplicity stratification $\{D_jU(n)\}_{j\le n}$ is {\it locally cone-like} in sense that
any $g\in D_jU(n)$ possesses a neighborhood $U$ in $U(n)$ equipped with a homeomorphism
\[
U\cong CL \times \Bbb R^{n^2-j^2}\, ,
\]
where $L$ is a compact filtered space of dimension $j^2-1$,  $CL$ is the open cone on $L$ and the homeomorphism preserves filtrations \cite{Siebenmann}. In our case, $L$ is the link of the point $I_j \in U(j)$.  
\end{rem}

The following result plays a key role of the proof of Theorem 
\ref{bigthm:indexthm}.

\begin{cor}\label{cor:mult_in_hom} There is an isomorphism
  \[
    H_{n^2}(D_1 U(n) \times S^1 \, | \, p) \cong  \Z^{m_p}\, .
  \]
\end{cor}

\begin{proof} This follows immediately from Proposition \ref{prop:mult_in_hom}
and the relative K\"unneth theorem.
    \end{proof}

\begin{figure} 
  \begin{tikzpicture}
    \draw[fill=gray!25!white,thick] (2,2) circle (2cm);
    \draw[fill=white,thick] (3,2) circle (1cm);
    \node at (4,2) [circle,draw=black,fill=black,inner sep=0.5mm, label=right:{$p$}] {};
    \draw[thick, dashed, draw=gray!50!black] (1,2) ellipse (1cm and 0.5cm) ;
    \draw[thick, rotate around={45:(3.3,3.2)},dashed, draw=gray!50!black] (3.28,3.2) ellipse (0.25cm and 0.13cm) ;
    \draw[thick, rotate around={-45:(3.3,0.8)},dashed, draw=gray!50!black] (3.28,0.8) ellipse (0.25cm and 0.13cm) ;
  \end{tikzpicture}
\caption{ \label{fig:rendering} 
A rendering of $D_1U(2) = \Sigma (\CP^1_+)$ in one dimension less. 
The point $p$ corresponds
to the identity matrix. The two small cones
 emanating from  $p$ generate the local homology (cf.\ the proof of Proposition \ref{prop:mult_in_hom}).}
\end{figure}
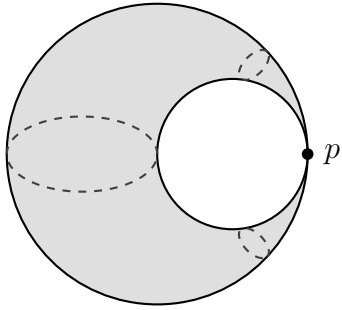

\begin{rem}\label{rem:add-to-prop} 
With respect to the isomorphism of 
Corollary \ref{cor:mult_in_hom} and the identification 
$H_{n^2}(D_1U(n) \times S^1)\cong \Z$, it is straightforward to check that
the evident homomorphism 
\[
H_{n^2}(D_1U(n) \times S^1) \to 
H_{n^2}(D_1U(n) \times S^1\, |\, p)
\]
 corresponds to the diagonal homomorphism $\Z \to \Z^{m_p}$.
\end{rem}

  

\begin{construct}
Suppose that 
\[
F\: (X,A) \to (Y,B)
\] is a map of pairs such that
$B \subset Y$ is a closed subspace and  $F^{-1}(B) = A \amalg A'$,
where $A'$ is disjoint from $A$. Suppose also that $X$ is a normal space. Then there is an
open neighborhood $U$ of $A$ which is disjoint from an open neighborhood of $A'$. It follows that $F$ determines
a map of pairs $(U\, |\, A) \to (Y\, |\, B)$. By excision, $H_\ast(U\, |\, A) \cong
H_\ast(X\, |\,A)$. Hence there is an induced homomorphism
 \[
 (F|A)_\ast\: H_\ast(X\, |\, A) \to 
H_\ast(Y\, |\, B)\, .
\]
\end{construct}

Let $p = (f(z), z) \in \mathfrak I_f$. Apply the construction to 
the map of pairs 
\[
j \times f\: (D_1U(n) \times S^1,p) \to
(U(n) \times U(n),\Delta)
\]
 to obtain a homomorphism 
\begin{equation} \label{eqn:j-times-gamma}
(j\times f|p)_*\:  H_{n^2}(D_1U(n) \times S^1 \, |\, p)  \ra 
    H_{n^2}(U(n) \times U(n) \,|\Delta)\, .
\end{equation}

Consider the composition
\begin{equation} \label{eqn:composition}
\xymatrix{
H_{n^2}(D_1U(n) \times S^1) \ar[r] & H_{n^2}(D_1U(n) \times S^1\, |\, p) 
\ar[d]^{(j\times f|p)_*} \\
& H_{n^2}(U(n) \times U(n) \,|\Delta)  \ar[r]^(.75){\langle x,{-}\rangle}_(.75)\cong
& \Bbb Z\, ,}
\end{equation}
where $x\in H^{n^2}(U(n) \times U(n) \,|\Delta)$ is the orientation class.
The unlabeled arrow is induced by the evident map 
$(D_1U(n) \times S^1\, |\, D_1U(n) \times S^1)\to (D_1U(n) \times S^1\, |\,	p)$.

\begin{prop}\label{prop:losum}  The homomorphism \eqref{eqn:composition}
maps the class $\mu \times [S^1]$ to
  the local intersection index $\iota_p$.
    \end{prop}

\begin{proof} The unlabeled homomorphism appearing in the displayed composition 
\eqref{eqn:composition}
is identified with the diagonal $\Z \to \Z^{m_p}$ (cf.\ Remark
\ref{rem:add-to-prop}). We  need to  identify the homomorphism $(j\times
f|p)_\ast$ of \eqref{eqn:j-times-gamma}, which is of the form $\Z^{m_p}\to
\Bbb Z$.  If $u_i$ generates the $i$-th factor of $\Z^{m_p}$, $1\le i \le m_p$,
then the value of this homomorphism on $u_i$ is an integer $s_i$. It 
suffices to show  $\sum_i s_i = \iota_p$. An artifact of the argument given below is
that $s_i \ne 0$ for at most two of the indices. 

Using the proof of Proposition~\ref{prop:mult_in_hom},
we can find a compact neighborhood $U$
of $g$ in $U(n)$ which is diffeomorphic to $D^{n^2}$ in such a way that
$ V := U \cap D_1U(n)$ is a contractible neighborhood
of $g$ in $D_1U(n)$ and $ V \setminus g$ retracts onto $\partial V \subset \partial U \cong S^{n^2-1}$.

  For $z \in S^1$ with $f(z) = g$,
  define  $W \subset S^1$ to be the closed arc
 given by taking the connected component of the preimage $f^{-1}(U)$ 
 that contains
 $z$. 
 Then $f$ restricts to give a map of pairs $(W,\partial W)\to (U,\partial U)$. 
  
Consider
  the homomorphism 
    \begin{equation} \label{eqn:local-version}
 (j\times f|p)_\ast\:   H_{n^2}(V \times W\, | \, p) \to 
    H_{n^2}(U(n) \times U(n) \, | \,\Delta)
 \end{equation}
 Clearly, it will be enough to identify \eqref{eqn:local-version} since it is a localized form of \eqref{eqn:j-times-gamma}.
 
  If $w\: W \to S^1$ is the inclusion, then the pushforward of the fundamental
  class of the circle, i.e.,
  \[
  \mu_W := w_{!}([S^1]) \in H_1(W\, |\, z) \cong \Bbb Z\, ,
\]
is a generator. Here we  orient $W$ using the counter-clockwise orientation of $S^1$.
  
  By Corollary~\ref{cor:mult_in_hom}, the domain of
  \eqref{eqn:local-version} is isomorphic to $\Z^{m_p}$, where
  $m_p$ is the multiplicity of $f(z)$.  
  If $p = (g,z) \in D_1U(n) \times S^1$, then following the proof
   of Proposition~\ref{prop:mult_in_hom}, one has preferred
  generators $\{\mu_i\}_{i =1,\ldots,m_p}$ for 
  $H_{n^2-1}(V|g)$.  Then 
 \[
 \{\mu_i\times \mu_W\}_{i =1,\ldots,m_p}
 \]
 is a set of generators for $H_{n^2}(V\times W|p)$.
 
 Applying
  the orientation class to  the image of these generators yields the
  intersection product $j_*(\mu_t) \cdot f_*(\mu_W)$. To analyze this product,
  we make use of the fact that the homological intersection product is equivalent
  to a linking number~\cite[p.~361,~Ex.~2,~3]{Spanier:AlgTop}. 
  Therefore, $j_*(\mu_t) \cdot f_*(\mu_W)$ is non-zero only when the curve $f(W)$ 
  links with $j_*(\mu_t)$, which occurs when $f(W)$ has non-trivial intersection with the
  path component of $S^{n^2-1} \setminus \partial V$ corresponding to
  the class $\mu_t$ via Alexander duality. 
  
Concretely, let $A \subset S^1$ be the open arc given by
 $\{e^{id}| -\delta < d < \delta\}$, 
 where $\delta > 0$ is chosen in such a way that for every $h\in U$, 
precisely $m_p$ of the eigenvalues of $h$ lie in $A$. 
Let $A_+ \subset A$ be the sub-arc
given by $\{e^{id}| 0 < d < \delta \}$.

Label the components of $S^{n^2-1} \setminus \partial V$ by
  $X(g)_j$ for $0 \le j \le m_p$, where $X(g)_j$ is characterized by matrices
  having exactly $j$ eigenvalues lying in $A_+$. 
  Then the curve $f(W)$ will enter
  $S^{n^2-1} \setminus \partial V$ at a component $X(g)_{k}$ 
  and leave it at a component $X(g)_{\ell}$ (with $k$ possibly equal to $\ell$). It follows that the linking number
  is given by $\ell - k$ (cf.~Figure~\ref{fig:nbhd_U}), which is the same thing 
  as the local intersection index $\iota_p$.  
\end{proof}

\begin{figure}
  \begin{tikzpicture}
    \draw (2,2) circle (3cm);
\shade[top color=gray!50!white, bottom color=gray!50!white, draw=black] (2.52,4.95) .. controls (2,3.1) .. (2,2) .. controls (1.3,2.5) .. (.97,4.82) arc (110:80:3cm);
\shade[top color=gray!50!white, bottom color=gray!50!white, draw=black] (.732,-.72) .. controls (1,-.5) .. (2,2) .. controls (2.5,1.1) .. (2.78,-.898) arc (285:245:3cm);
\draw (5,2) .. controls (3,2.5) .. (2,2);
\draw (4.12, 4.12) .. controls (3,3) .. (-0.12, -0.12);
\draw (-.6,3.5) .. controls (0,2.5) .. (2,2);
\draw (-.75,.78) .. controls (0,2) .. (2,2);
\begin{scope}[very thick, red, dashed, decoration={
markings, mark=at position 0.5 with {\arrow{latex}}}]
\draw[postaction={decorate}]  (-1,2) .. controls (-.5, 2.3) .. (2,2);
\draw[postaction={decorate}] (2,2) .. controls (4,1) .. (3.93, -.3);
\end{scope}
\node at (-1,2) [circle,draw=red,fill=red,inner sep=1mm] {};
\node at (3.93,-.3) [circle,draw=red,fill=red,inner sep=1mm] {};
\node at (2,2) [circle,draw=black,fill=black,inner sep=0.7mm] {};
\node at (5,0) {$X(g)_{j_1}$};
\node at (5.5,3) {$X(g)_{j_2}$};
\node at (3.5,5) {$X(g)_{j_3}$};
\node at (0,4.9) {$X(g)_{j_4}$};
\node at (-1.8,2) {$X(g)_{j_5}$};
\node[label={[rotate=30]$\vdots$}] at (-1.3,0.4) {};
  \end{tikzpicture}
\caption{A picture of the neighborhoods $U$ and $V$ of $g \in D_1U(n)$
(cf.~ proof of Proposition \ref{prop:losum}).
The center point is $g$, the entire region is $U$, and the shaded portion is $V$. The arc $f(W)$
is shown in red, and the various components are labelled by $X(g)_j$ for
$0 \le j \le m_p$. In this example, the curve enters
the neighborhood at $X(g)_{j_5}$ and leaves it at
$X(g)_{j_1}$, and so the local intersection index
at $g$ is $j_1 - j_5$.}
\label{fig:nbhd_U}
\end{figure}
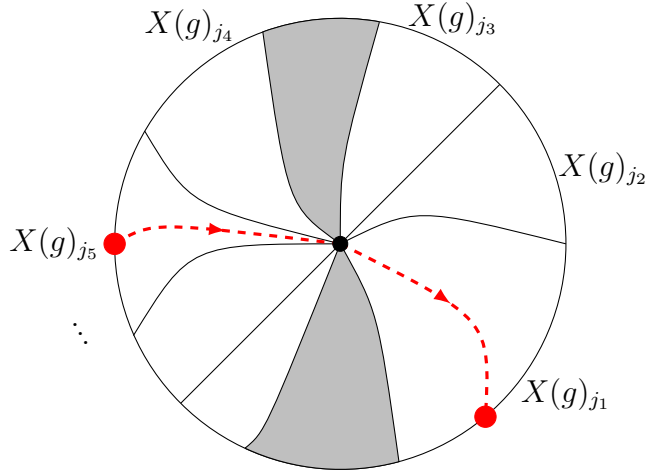

\section{The proof of Theorem \ref{bigthm:indexthm}}
\label{sec:proof-index-thm}

The proof of Theorem \ref{bigthm:indexthm}  relies on the commutative diagram
\begin{equation*}
  \xymatrixcolsep{3pc}\
  \xymatrix{
    H_{n^2}(D_1U(n) \times S^1) \ar@/_8.0pc/[dd]_{a_\ast}
    \ar[r]^{(j \times f)_\ast}
    \ar[d]
    &
   H_{n^2}(U(n) \times U(n))
    \ar[d]^{b_\ast}
    \\
    H_{n^2}(D_1 U(n) \times S^1\, |\, \mathfrak I_{f}) \ar[d]_\cong
    \ar[r]^{(j \times f|\mathfrak I_f)_\ast}
    &    H_{n^2}(U(n) \times U(n)\, |\, \Delta)
        \\
  \bigoplus_{p \in \mathfrak I_{f}} H_{n^2}(D_1U(n) \times S^1\, | \, p) \ar[ur]_{\oplus (j \times f|p)_\ast} \ar[ur] 
  }
\end{equation*}
where the $p$-th component of the arrow $a_\ast$ is induced by the evident map
\[
(D_1U(n) \times S^1|D_1U(n) \times S^1) \to (D_1U(n) \times S^1|p)\, .
\] 
and $b_\ast$ is as in Lemma \ref{lem:globalint}.

By Lemma~\ref{lem:globalint}, 
the composite $b_\ast\circ (j\times f)_\ast$ maps $\mu \times [S^1]$ to 
a class which, when paired with the orientation class, yields $\alpha_f$.
By Proposition \ref{prop:mult_in_hom}, Proposition ~\ref{prop:losum}, and Remark \ref{rem:add-to-prop},
the composite $\oplus (j \times f|p)_\ast\circ a_\ast$
maps  $\mu \times [S^1]$ to a class which pairs to  $q = \sum \iota_p$.
Invoking the commutativity of the diagram completes the proof.\qed

\section{The proof of Theorem \ref{bigthm:lower-bound}} 
\label{sec:lower-bound}
Recall the scattering matrix $\Gamma^a \: S^1 \to U(Y^a_1)$ defined
in \S\ref{sec:intro}.  

\begin{defn} The 
  {\em winding number} $w(\Gamma^{a})$ of a vertex $a \in Y_0$ is the 
  degree of the composition
\[
  S^1 @>\Gamma^{a}>> U(Y^a_1) @> \det >> U(1) \cong S^1\, .
\]
\end{defn}
The following proposition
identifies the global intersection index in terms of the winding
numbers and the length function $L$. 

\begin{prop}\label{prop:glwind} 
  \[
    \alpha_\Gamma =  \sum_{ab\in X_1} L_{ab} +
    \sum_{a \in X_0} w(\Gamma^{a})\,  .
  \] 
\end{prop}

\begin{proof} 
By Poincar\'e duality, the class 
$\Gamma_*([S^1]) \in H_1(U(n))$ is given by  $\alpha_\Gamma\iota$,
where $\iota$ is the generator which hits the fundamental class of $S^1$ under
$\text{\rm det}_\ast\: H_1(U(n)) \to H_1(S^1)$.  
The result now follows from the calculation
\[
\text{\rm det}_\ast\Gamma_\ast[S^1] = \text{\rm det}_\ast(\alpha_\Gamma\iota)
= \alpha_\Gamma\text{\rm det}_\ast(\iota) = \alpha_\Gamma [S^1]\, .
\]
The displayed formula
involves a computation of the determinant of  
$\Gamma$. For $z = e^{ik}\in S^1$, the homomorphism property of the determinant 
implies the equation
\begin{equation} \label{eqn:degree_calc}
\det \Gamma(z)  = 
\det e^{ik\hat L} \cdot  \det \Gamma_0(z)\, .
\end{equation}
The number $\alpha_\Gamma$  is therefore
 just the degree of the map $\det \Gamma\: S^1 \to S^1$, and is given by the sum of the degrees of the two maps appearing on the right side of \ eqn.~\eqref{eqn:degree_calc}.

Since $\hat L$ is a diagonal operator, the term $\det e^{ik\hat L}$ 
is given by $\prod_{ab}e^{ik L_{ab}}$. Its degree is therefore
the sum $\sum_{ab\in X_1} L_{ab}$. 
The determinant of $\Gamma_0(k)$ is the product of $\det w(\Gamma^a)$ indexed over the vertices of $X$, so its degree is  $ \sum_{a \in X_0} w(\Gamma^{a})$.
It follows that the degree of $\det\Gamma\: S^1 \to S^1$ is given by the expression
\[
 \sum_{ab\in X_1} L_{ab} +
  \sum_{a \in X_0} w(\Gamma^{a})\, .\qedhere
 \] 
\end{proof}

\begin{proof}[Proof of Theorem \ref{bigthm:lower-bound}]
By the inequality \eqref{eqn:m>p} in conjunction with Theorem \ref{bigthm:indexthm}, we infer
\[
m \,\, \ge\,\, q = \alpha_{\Gamma}\, .
\]
The result now follows by Proposition \ref{prop:glwind}.
 \end{proof}

\section{The proof of Addendum \ref{bigthm:long-arm}}
\label{sec:long-arm}

Let $\zeta = e^{i\ell} \in S^1$ be a point for which $\Gamma(\zeta) :=\tilde \Gamma(\ell)$ has an eigenvalue equal to +1. Set 
\[
p := (\Gamma(\zeta),\zeta) \in \mathfrak I_\Gamma\, .
\]
Set $f(k) := e^{ik \hat L}$. Taking the first Taylor approximation of $\tilde \Gamma$ at $\ell$,
we obtain
\begingroup
\begin{align}\label{eqn:Taylor} 
  \tilde \Gamma(k) \,\, &= \,\, \tilde\Gamma(\ell) + 
  (k{-}\ell)f'(\ell)  \tilde \Gamma_0(\ell)  +  f(\ell)
  \tilde \Gamma_0'(\ell))
+ o\left( (k{-}\ell)^2 \right)\, , \nonumber\\
 &= \,\, \tilde\Gamma(\ell) + 
  (k{-}\ell)f'(\ell)  \tilde \Gamma_0(\ell)  +  
    (k{-}\ell)f(\ell) 
  \tilde \Gamma_0'(\ell)
+ o\left( (k{-}\ell)^2 \right)\, ,
\end{align}
\endgroup
for $|k-\ell|< \delta$ and for some
 $\delta >0$ sufficiently small.
Observe that the first term $\tilde\Gamma(\ell)$ is length independent (since $|e^{i\ell\hat L}| = 1)$. 
The  third term, which contains $\tilde \Gamma_0'(\ell)$,
measures the dependence of scattering on $k$, which is a local
phenomenon. It too is length independent. 
On the other hand, the second term 
\[
(k{-}\ell)f'(\ell)  \tilde \Gamma_0(\ell) \,\, =\,\, (k-\ell) i \hat  L \tilde \Gamma(\ell)
\]
is length dependent. When the edge lengths are large, the second term of the approximation 
dominates the third.

Hence, we omit the third term
of \eqref{eqn:Taylor} to obtain the approximation, valid up to the second order
\[ \tilde \Gamma(k) \,\, \approx\,\, (I + i(k-\ell) \hat L)\tilde \Gamma(\ell)\,
\]
for $|k -\ell| < \delta$ and sufficiently large edge lengths.
Since 
$\hat L$ is given by rescaling the edge lengths, we see that 
each of its eigenvalues must exceed the minimal edge length, which we denote by
$\mu$, of the weighted graph.
Consequently, if the edge lengths are sufficiently large, the eigenvalues of the 1-jet  $J_{\ell}^1\tilde \Gamma(k)$ are 
bounded below by the operator
\[
(I + i(k-\ell)\mu)\tilde \Gamma(\ell)\, .
\]
We infer that 
the imaginary part of any eigenvalue of 
 $J_{\ell}^1\tilde \Gamma(k)$ is positive if $k > \ell$
 and negative  if $k < \ell$ for $k$ lying in a sufficiently small neighborhood
 of $\ell$.  
 
Hence, ignoring the distinction between $ \tilde \Gamma(k)$ and 
$J_{\ell}^1\tilde \Gamma(k)$, it follows
by Definition \ref{defn:iota-pm} that $\iota_p^- = 0$
and 
\[
\iota_p = \iota_p^+ = m_p
\]
when the edge length is large.
Summing up over all $p$ completes the proof.\qed

\end{document}